\documentclass[12pt,reqno]{amsart}
\usepackage{amsfonts,amssymb,amsmath,amsthm}
\usepackage{comment}
\usepackage[english]{babel}
\newtheorem{thm}{Theorem}[section]
\newtheorem{cor}[thm]{Corollary}
\newtheorem{lem}[thm]{Lemma}
\newtheorem{prop}[thm]{Proposition}
\theoremstyle{definition}
\newtheorem{defn}[thm]{Definition}
\newtheorem{rem}[thm]{Remark}
\newtheorem{ex}[thm]{Example}
\numberwithin{equation}{section}

\newcommand{\Z}{\mathbb Z}

\newcommand{\id}{\mathrm{id}}

\newcounter{zlist} 
  \newenvironment{zlist}{\begin{list}{(\arabic{zlist})}{ 
  \usecounter{zlist}\leftmargin2.5em\labelwidth2em\labelsep0.5em 
  \topsep0.6ex
  \parsep0.3ex plus0.2ex minus0.1ex}}{\end{list}}

  \newcounter{blist} 
  \newenvironment{blist}{\begin{list}{(\alph{blist})}{ 
  \usecounter{blist}\leftmargin2.5em\labelwidth2em\labelsep0.5em 
  \topsep0.6ex  
  \parsep0.3ex plus0.2ex minus0.1ex}}{\end{list}} 
  
  \makeatletter
\@namedef{subjclassname@2020}{\textup{2020} Mathematics Subject Classification}
\makeatother

\title{Lie and Nijenhuis brackets on affine spaces}
\author[T. Brzezi\'nski]{Tomasz Brzezi\'nski}

\address{
Department of Mathematics, Swansea University, 
Swansea University Bay Campus,
Fabian Way,
Swansea,
  Swansea SA1 8EN, U.K.\ \newline 
Faculty of Mathematics, University of Bia{\l}ystok, K.\ Cio{\l}kowskiego  1M,
15-245 Bia\-{\l}ys\-tok, Poland}

\email{T.Brzezinski@swansea.ac.uk}
\author[J. Papworth]{James Papworth}

\address{
Department of Mathematics, Swansea University, 
Swansea University Bay Campus,
Fabian Way,
Swansea,
  Swansea SA1 8EN, U.K.}

\email{918550@swansea.ac.uk}


\begin{document}
\begin{abstract}
Lie algebras are extended to the affine case using the heap operation, giving them a definition that is not dependent on the unique element $0$, such that they still adhere to antisymmetry and Jacobi properties. It is then looked at how Nijenhuis brackets function on these Lie affgebras and demonstrated how they fulfil the compatibility condition in the affine case.
 \end{abstract}
\maketitle

\section{Introduction} \noindent
The aim of this paper is twofold: to present an intrinsic definition of a Lie bracket on an affine space that does not depend on any choice of a vector space (or a point in the affine space) and to extend the theory of Nijenhuis operators from Lie algebras to the affine case. 

A need for considering Lie brackets on affine spaces arose naturally as a consequence of the Tulczyjew programme of frame-independent approach to classical mechanics \cite{Tul:fra}. The differential geometric aspects of the programme involved discussion of structures in which vector spaces are replaced by affine spaces \cite{GraGra:fra} and led to the development of the AV-differential geometry in\cite{GraGra:av1}, \cite{GraGra:av2}. The appearance of Lie brackets on affine spaces or {\em Lie affgebras} in \cite{GraGra:Lie} was a necessary step. In their definition of a Lie affgebra the authors of \cite{GraGra:Lie} rely on the existence of a vector space: both the antisymmetry and the Jacobi identity of a Lie bracket are formulated in this vector space. On the other hand a vector space independent definition of an affine space is available (see e.g.\ \cite{OstSch:bar} or \cite{BreBrz:hea} for a more recent contextual discussion). This uses the heap ternary operation \cite{Pru:the} naturally embedded in an affine space and, also ternary, multiplication by scalars (see Definition~\ref{def.aff.mod}). One can assign a vector space to \textbf{any} element of the affine space in which the chosen point is the zero vector. The intrinsic definition of a Lie affgebra we propose in this paper, Definition~\ref{def.Lie},  removes the need for a specified element entirely and formulates the antisymmetry and Jacobi identity of the Lie bracket without invoking a neutral element or a vector space. In Theorem~\ref{thm.Lie.alg} we show that once a choice of a point is made and vector space obtained then such an affine Lie bracket reduces to a Lie bracket on this vector space through linearisation; non-isomorphic brackets on the affine level might produce isomorphic or even identical ones on the linear level (see Example~\ref{ex.action.o}). 

An important method of solving integrable Hamiltonian systems relies on the existence of (at least) two Poisson structures, compatible one with the other in the sense that their linear combination is again a Poisson structure \cite{Mag:sim}. This leads to the bi-Hamiltonian structure: the same mechanical process is governed by two Hamiltonians, which allows one to find easier all constants of motion in involution. One of the ways of ensuring that compatible Poisson brackets exist is to deform the given Poisson structure in a way inspired by the Nijenhuis-Richardson deformation of graded Lie algebras \cite{NijRic:def}. This has been explored in \cite{KosSchMag:Poi} in terms of Poisson-Nijenhuis structures which arise trough the existence of a particular operator, termed a {\em Nijenhuis operator}. As presented in \cite{KosSchMag:Poi}, Nijenhuis operators can be defined for all Lie algebras, not only for Poisson algebras. It was observed in \cite{BrzPap:aff} in the associative case that the conditions that Nijenhuis operators and Nijenhuis brackets need to satisfy are affine in flavour and hence compatible with the structures defined on affine spaces. In this paper we extend Nijenhuis brackets to Lie affgebras of Definition~\ref{def.Lie}. Nijenhuis brackets and operators do not rely on any choices of elements, but once a choice is made and a Lie affgebra reduced to the Lie algebra, the corresponding linearisation of the Nijenhuis operator is a Nijenhuis operator on this Lie algebra in the sense of \cite[Definition~1.1]{KosSchMag:Poi}.  One  key result of the present paper is the affine (or barycentric) compatibility of the initial Lie and deformed Nijenhuis brackets Theorem~\ref{thm.Nij.iter}. This would allow one to use methods of bi-Hamiltonian systems in frame independent mechanics. 

We begin with preliminary information on heaps (Definition \ref{def.heap}), affine $\mathbb{K}$-modules (Definition \ref{def.aff.mod}), and $\mathbb{K}$-affgebras (Definition \ref{def.assoc}) as well as a remark on the translation isomorphism that will be used in the latter section of the paper (Remark \ref{rem.trans}).

Then using these notions, we form a new definition, the Lie affgebra (Definition \ref{def.Lie}), a way of interpreting the standard Lie bracket
as a bi-affine operation as opposed to a bilinear one. This structure has its own version of the antisymmetry and Jacobi identity conditions, but these condition no longer depend on the existence of the unique neutral element. We then go through various examples and properties of Lie affgebras throughout Section 3 as well as showing how the affine Lie bracket reduces to the linear case in Proposition \ref{prop.action}, pre-Lie affgebras (Definition \ref{def.pre}), and the interaction between Lie affgebras and derivations (Definition \ref{def.deriv}).

Finally, in Section 4 we look to how Nijenhuis operators act on Lie affgebras, taking some definitions from  \cite{KosSchMag:Poi} and generalising them to the affine case. In Definition~\ref{def.nijen.op} we redefine both the Nijenhuis operator and the Nijenhuis bracket so that they no longer depend on the existence of the unique element $0$, and can thus be used in our affine calculations. In analogy to the Lie algebra case, Theorem~\ref{nij.is.lie.bra} states that an affine Nijenhuis bracket is a Lie bracket (in the sense of Defition~\ref{def.Lie})  when it is acting on a Lie affgebra. We then establish that any power of a Nijenhuis operator is a Nijenhuis operators and that the corresponding brackets are affine or barycentic compatible in Theorem~\ref{thm.Nij.iter}. The paper then closes with the reduction of affine Nijenhuis operators to linear ones in Theorem~\ref{thm.nij.final}.

\section{Preliminaries on heaps, affine spaces and affgebras}
The following notion was proposed in \cite{Pru:the} as an affine approach to abelian groups.  
\begin{defn}\label{def.heap} 
An \textbf{abelian heap} is a set $A$, together with a ternary operation 
$$\langle---\rangle:H^{3} \rightarrow H, \qquad (a,b,c) \mapsto \langle a,b,c\rangle,$$ such that for all $a,b,c,d,e \in H$, 
$$ 
    \langle\langle a,b,c\rangle,d,e\rangle = \langle a,b, \langle c,d,e\rangle \rangle, \quad 
 \langle a,b,b\rangle = a = \langle b,b,a\rangle, \quad \langle a,b,c\rangle = \langle c,b,a\rangle
$$
A homomorphism of heaps is a function that preserves the heap operations.
\end{defn}

An abelian heap is an affine version of an abelian group in the sense that first any abelian group is an abelian heap with operation $\langle a,b,c\rangle = a-b+c$, second, any coset of an abelian group is an abelian heap, and, third, any non-empty abelian heap $H$ gives rise to the family of isomorphism groups by {\em retracting} at any element, that is by reducing the ternary operation to the binary one by setting $a+b = \langle a,o,b\rangle$ for any fixed element $o\in H$. One easily finds that with this definition of addition, 
\begin{equation}\label{+-}
\langle a,b,c\rangle = a-b+c,
\end{equation}
which explains that a heap operation can always be understood as a combination of addition and subtraction and also justifies our frequent interpretation of this operation as such a combination\footnote{One has to be constantly aware  about the meaning of the binary operations as they are defined relatively to a chosen element $o$, even though this choice is immaterial for the overall value of the combination \eqref{+-}}. This is also a quick way of realising that the placements of brackets $\langle, \rangle$ in multi-application of the heap operation does not play any role as long as the parity of position of elements is preserved. Thus we will write
$$
\langle a_1,a_2,\ldots, a_{2n+1}\rangle _{2n+1} := 
\langle \ldots  \langle \langle a_1,a_2,a_3\rangle,a_4,a_5\rangle, \ldots, a_{2n+1}\rangle,
$$
for multiple application of the ternary operation, or simply $
\langle a_1,a_2,\ldots, a_{2n+1}\rangle,
$ if the length of the expression is clear from the context (which it is in this case). Finally, any equality of the form
$$
\langle a,b,c \rangle = d
$$
remains valid under a cyclic permutation $a\to b\to c\to d \to a$.

Abelian heaps can be also understood as affine $\mathbb{Z}$-modules in the sense of \cite{OstSch:bar}. We explain this presently, first by viewing the original definition of affine modules through lenses of heaps of modules \cite{BreBrz:hea}, which allows one to make an {\em intrinsic}, i.e.\ with no reference to a vector space, definition of an affine space.

\begin{defn}\label{def.aff.mod}
    Let $\mathbb{K}$ be a commutative ring with identity. Then an \textbf{affine $\mathbb{K}$-module} is a non-empty abelian heap $A$ together with the ternary action
    $$
    \triangleright : \mathbb{K}\times A\times A \longrightarrow A, \qquad (\alpha,a,b)\longmapsto \alpha \triangleright_a b, $$
    satisfying the following conditions, for all $a,b,c\in A$ and $\alpha,\beta \in \mathbb{K}$,
    \begin{blist}
        \item $\alpha \triangleright_a - : A\to A$ and $- \triangleright_a b : \mathbb{K}\to A$ are heap homomorphisms, where $\mathbb{K}$ is understood as a heap with the operation $\alpha-\beta +\gamma$;
        \item $(\alpha\beta) \triangleright_a b = \alpha \triangleright_a (\beta \triangleright_a b)$;
        \item $\alpha\triangleright_a b = \langle \alpha\triangleright_c b, \alpha\triangleright_c a, a\rangle$;
        \item $0\triangleright_a b = a$ and $1\triangleright_a b = b$.
    \end{blist}
    The element $a$ in $\alpha \triangleright_a b$ is referred to as the \textbf{base} of the action. A homomorphism of affine $\mathbb{K}$-modules is a heap homomorphism $f$ such that
    $$
    f(\alpha\triangleright_a b) = \alpha\triangleright_{f(a)} f(b).
    $$

    As is customary,  when $\mathbb{K}$ is a field we refer to an 
    \textbf{affine space} rather than an affine module.
\end{defn}

It is worth to note that the combination of conditions (a) and (c) implies that, for all $a,b,c,d\in A$ and $\alpha\in \mathbb{K}$,
\begin{equation}\label{heap.middle}
 \alpha \triangleright_{\langle a, b,c\rangle} d = \langle \alpha \triangleright_a d, \alpha \triangleright_b d, \alpha \triangleright_c d\rangle,   
\end{equation}
i.e. $\alpha\triangleright_{-} d$ is a heap homomorphism; see \cite[Lemma~3.5]{BreBrz:hea}. Furthermore, (c) alone implies that 
\begin{equation}\label{aa}
    \alpha \triangleright_a a = a.
\end{equation}

One easily checks that an abelian heap $A$ is an affine $\Z$-module with the unique action determined by conditions (d) and (a) in Definition~\ref{def.aff.mod}:
$$
n\triangleright_a b = \begin{cases}
    \langle b,a,b,\ldots,a, b\rangle_{2n-1} & n>0,\\
    \langle a,b,a,\ldots,b, a\rangle_{|2n+1|} & n\leq 0.
\end{cases}$$
Any morphism of heaps automatically is a morphism of corresponding affine $\Z$-modules. This establishes a categorical isomorphism of affine $\Z$-modules and abelian heaps in the same way as isomorphism of the categories of abelian groups and $\Z$-modules.

We will make frequent use of the following,
\begin{lem}\label{lem.aff.map}
Let $f, g,h: A\to B$ be homomorphisms of affine $\mathbb{K}$-modules. Then 
$$
\langle f,g,h\rangle : A\to B, \qquad a\mapsto \langle f(a),g(a),h(a)\rangle,
$$
is a homomorphism of affine $\mathbb{K}$-modules
\end{lem}
\begin{proof}
    The lemma is a consequence of the base change property (c) in Definition~\ref{def.aff.mod}. More precisely, for all $\alpha\in \mathbb{K}$ and $a,b\in A$, the actions
    $$
    \begin{aligned}
      \langle f,g,h\rangle(\alpha\triangleright_ab) = \langle \alpha\triangleright_{f(a)}f(b),\alpha\triangleright_{g(a)}g(b),\alpha\triangleright_{h(a)}h(b)\rangle,  
    \end{aligned}
    $$
    can be brought to a common base, say $f(a)$ to give
    $$
    \begin{aligned}
      \langle \alpha\triangleright_{f(a)}f(b),\alpha\triangleright_{f(a)}g(b), \alpha\triangleright_{f(a)}g(a), g(a), \alpha\triangleright_{f(a)}h(b), \alpha\triangleright_{f(a)}h(a), h(a)\rangle.  
    \end{aligned}
    $$
    The rearrangements and the heap homomorphism properties (a) in Definition~\ref{def.aff.mod} combined with \eqref{aa}  and the base change prooperty (c) again yield
    $$
    \Big\langle \alpha\triangleright_{f(a)}\langle f,g,h\rangle(b), \alpha\triangleright_{f(a)}\langle f,g,h\rangle(a),  \langle f,g,h\rangle(a)\Big\rangle = \alpha\triangleright_{\langle f,g,h\rangle(a)}\langle f,g,h\rangle(b),
    $$
    as required.
\end{proof}

\begin{rem}\label{rem.act.+-}
    Let $A$ be an affine $\mathbb{K}$-module. We can take any element $o\in A$ and first consider the abelian group $A_o$. Thanks to the base change axiom (c) in Definition~\ref{def.aff.mod}, the bases of all actions can be related to the base $o$. Writing $\alpha a := \alpha\triangleright_o a$ we then immediately find that in terms of operations in $A_o$,
    \begin{equation}\label{act.+-}
    \alpha\triangleright_a b = \alpha b - \alpha a +a, \qquad \mbox{for all $a,b\in A$, $\alpha \in \mathbb{K}$}.
    \end{equation}
    Since the left-hand side of \eqref{act.+-} is independent of the choice of $o$, the right-hand side is also independent of this choice in the sense that the expression formally remains the same whatever  $o$ is chosen. Of course the meaning of binary operations is relative to this choice.
\end{rem}
Let us explain briefly that the definition of an affine space presented here is equivalent to the traditional definition that involves a set $A$ with a free and transitive action $+$ of a  $\mathbb{K}$-vector space $\overset{\rightarrow}{A}$. In particular, for all $b,c\in A$, there is a unique vector (from $b$ to $c$), and every vector in $\overset{\rightarrow}{A}$ can be obtained in this way. Furthermore, any point $a$ can be uniquely shifted or translated to another point $a + \overset{\longrightarrow}{bc}\in A$. This combination of three points equips $A$ with the abelian heap structure 
\begin{equation}\label{aff.heap}
    \langle a,b,c\rangle = a + \overset{\longrightarrow}{bc}. 
\end{equation}
Now $A$ becomes an affine $\mathbb{K}$-module in the sense of Definition~\ref{def.aff.mod} with the action
\begin{equation}\label{aff.act}
    \alpha \triangleright_a b  = a + \alpha \overset{\longrightarrow}{ab}. 
\end{equation}
Conversely, starting with an affine $\mathbb{K}$-space $A$ as in Definition~\ref{def.aff.mod} one can reconstruct the underlying vector space by picking any $o\in A$ and defining the addition of vectors and multiplication by scalars by 
\begin{equation}\label{aff.vec}
   a+b = \langle a, o,b\rangle, \qquad \alpha a = \alpha \triangleright_o a.
\end{equation}
The resulting vector space (or $\mathbb{K}$-module in case of a commutative ring rather than a field) will be denoted by $A_o$.
In this set-up, $\overset{\longrightarrow}{bc} = \langle o,b,c\rangle \in A_o$. 

\begin{rem}\label{rem.trans}
    Different choices of the zero  element in an affine $\mathbb{K}$-module $A$ lead to isomorphic $\mathbb{K}$-modules. More precisely, given $o,u\in A$, the map
    \begin{equation}\label{trans}
        \tau_o^u: A\longrightarrow A, \qquad a\longmapsto \langle a, o,u\rangle, 
    \end{equation}
    is an automorphism of heaps and an isomorphism of abelian groups $A_o\cong A_u$, known as the \textbf{translation isomorphism} (the inverse being $\tau_u^o$). Furthermore, for all $a\in A$ and $\alpha\in \mathbb{K}$,
    $$
    \begin{aligned}
        \alpha\triangleright_u\tau_o^u(a) &= \langle\alpha\triangleright_u a,\alpha\triangleright_u o, \alpha\triangleright_u u\rangle\\
        &= \langle\alpha\triangleright_u a,\alpha\triangleright_u o, o,o, u\rangle = \langle\alpha\triangleright_o a, o,  u\rangle =\tau_o^u\left(\alpha\triangleright_o a\right).
    \end{aligned}
    $$
    The first equality follows from the fact that $\alpha\triangleright_u -$ is a heap homomorphism, the second is a consequence of the Mal'cev identity and \eqref{aa}, and the penultimate equality follows by the base exchange property (c) in Definition~\ref{def.aff.mod}. This proves that $\tau_o^u$ is an isomorphism of $\mathbb{K}$-modules, $A_o\cong A_u$.
\end{rem}

If $\mathbb{K}$ is a field, then the \textbf{dimension} of an affine space $A$, written $\dim A$, is defined as the dimension of the vector space $A_o$. In view of Remark~\ref{rem.trans} the dimension does not depend on the choice of $o\in A$. 

Traditional definition of an affine map involves a pair, a function $f:A\to B$ and a (uniquely defined) linear transformation $\overset{\rightarrow}{f}:\overset{\rightarrow}{A}\to \overset{\rightarrow}{B}$, called the \textbf{linearisation} of $f$ such that $\overset{\rightarrow}{f}\left(\overset{\longrightarrow}{ab}\right) = \overset{-\!-\!-\!-\!-\!\longrightarrow}{f(a)f(b)}$. This is equivalent to say that $f$ is a homomorphism of affine spaces in the sense of Definition~\ref{def.aff.mod}. If we fix $o_A\in A$ and $o_B\in B$ and define the vector space structures $A_{o_A}$ and $B_{o_B}$as in \eqref{aff.vec}, then the linearisation of $f:A\to B$ is given by 
\begin{equation}\label{linearisation}
\overset{\rightarrow}{f} : A_{o_A} \to B_{o_B}, \qquad a\mapsto \langle f(a),f(o_A),o_B\rangle.
\end{equation}

\begin{defn}\label{def.assoc}
Let $\mathbb{K}$ be a commutative ring. An \textbf{(associative) $\mathbb{K}$-affgebra} or simply an \textbf{affgebra} is an affine $\mathbb{K}$-module $A$ together with a bi-affine (associative) multiplication $A\times A\to A$.
\end{defn}

The bi-affine property of the multiplication in an affgebra implies in particular that it is a bi-heap homomorphism, so an affgebra $A$ is also a truss \cite{Brz:tru}, \cite{Brz:par}, i.e.\ the following distributive laws hold
$$
a\langle b,c,d\rangle = \langle ab,ac,ad\rangle \quad \mbox{and}\quad \langle b,c,d\rangle a = \langle ba,ca,da\rangle,
$$
 for all $a,b,c,d\in A$. In view of the preceding interpretation of heaps as affine $\Z$-module, trusses are simply associative $\Z$-affgebras in the same was as rings are associative $\Z$-algebras.

 \begin{ex}\label{ex.affg}
   Let $A$ be an an affine $\mathbb{K}$-module. The set $\mathrm{Aff}(A)$ of all affine endomorphisms of $A$ is an affgebra with the pointwise heap operation,
   the action
   \begin{equation}\label{end.act}
     (\alpha \triangleright _f g )(a)=  \alpha \triangleright _{f(a)} g (a),
   \end{equation} 
   for all $\alpha \in \mathbb{K}$, $f,g\in \mathrm{Aff}(A)$, $a\in A$, and multiplication given by composition.
 \end{ex}
 \begin{proof}
     By Lemma~\ref{lem.aff.map}, $\mathrm{Aff}(A)$ is a heap and the general properties of endomorphism trusses affirm that $\mathrm{Aff}(A)$ is a truss. Since the action is defined pointwise, it satisfies the requirements of Definition~\ref{def.aff.mod}. It remains only to check that $\alpha \triangleright _f g$ is an affine map. That it is a heap homomorphism follows from \eqref{heap.middle}. By \cite[Lemma~3.15]{BreBrz:hea}, for all $\alpha,\beta \in \mathbb{K}$ and $w,x,y,z\in A$,
     $$
     \alpha\triangleright_{\beta\, \triangleright_x y}(\beta \triangleright_z w) = \beta\triangleright_{\alpha\, \triangleright_x z}(\alpha \triangleright_y w).
     $$
    Setting $x=f(a)$, $y=f(b)$, $z=g(a)$ and $w=g(b)$ for any $a,b\in A$ and $f,g\in \mathrm{Aff}(A)$ we obtain
    $$
    (\alpha\triangleright_f g)(\beta\triangleright_a b) = \beta\triangleright_{(\alpha\, \triangleright_f g)(a)} (\alpha \triangleright_f g)(b).
    $$
    This means that $\alpha\triangleright_f g$ is an affine map, and completes the proof of the example.
 \end{proof}
\section{Lie affgebras}\label{sec.Lie}
In this section we start by proposing the intrinsic definition of a Lie affgebra or a Lie bracket on an affine space. Next we give a number of examples of Lie affgebras, compare them to the structure introduced earlier in \cite{GraGra:Lie} and show the reduction to Lie algebras.

\begin{defn}\label{def.Lie}
   Let $A$ be an affine $\mathbb{K}$-module. A \textbf{left Lie bracket} on $A$ is a bi-affine map $[-,-]:A\times A\to A$ such that, for all $a,b,c\in A$,
   \begin{subequations}\label{left.Lie}
   \begin{equation}\label{anti}
   \big\langle [a,b],[a,a],[b,a]\big\rangle = [b,b],   
   \end{equation}
   \begin{equation}\label{l.Jacobi}
       \Big\langle [a,[b,c]], [a,a], [b,[c,a]], [b,b], [c,[a,b]]\Big\rangle = [c,c]
   \end{equation}
   \end{subequations}
   The bi-affine map $[-,-]:A\times A\to A$ is a \textbf{right Lie bracket} if it satisfies \eqref{anti} and 
   $$
   \Big\langle [[a,b],c], [a,a], [[b,c],a], [b,b], [[c,a],b]\Big\rangle = [c,c].
   $$
   An affine $\mathbb{K}$-module with the left (resp.\ right) Lie bracket is called a \textbf{left} (resp. \textbf{right}) \textbf{Lie affgebra}.

   In case that $\mathbb{K}=\Z$ (and hence the affine structure of $A$ is uniquely determined by the heap structure), we refer to a \textbf{Lie truss} rather than Lie $\Z$-affgebra.
\end{defn}

Whatever is being said about left Lie brackets and left Lie affgebras can be equivalently (symmetrically) said about right Lie brackets and affgebras. Therefore, in what follows we concentrated on the left case and drop the adjective {\em left} from the terminology.

\begin{rem}\label{rem.Lie}
  Note that the antisymmetry property \eqref{anti} can be equivalently stated as 
  $$
  \big\langle [a,b],[b,b],[b,a]\big\rangle = [a,a].
  $$
  Similarly, the order in which expressions $[a,a]$, $[b,b]$ and $[c,c]$ appear in the Jacobi identity \eqref{l.Jacobi} does not matter.
\end{rem}

\begin{ex}\label{ex.trivial}
    Let $A$ be an affine $\mathbb{K}$-module $A$ and let $\sigma:A\to A$ be an endomorphism of affine modules. Define the binary operation  $[-,-]: A\times A\to A$ by $[a,b] =\sigma(a)$. Since both the constant map and $\sigma$ are affine tranformations, this operation is bi-affine. That $[-,-]$ is a (left) Lie bracket is easily verified.  

    Now, let us consider two endomorphisms $\sigma_1, \sigma_2$. Then the corresponding Lie affgebras are isomorphic if and only if there exists an automorphism of affine spaces $f:A\to A$ such that, for all $a,b\in A$,
    $$
    \sigma_2(f(a)) = [f(a),f(b)]_2 = f([a,b]_1) = f(\sigma_1(a)),
    $$
    that is, if and only if
    $$
    \sigma_2 = f\circ \sigma_1 \circ f^{-1}.
    $$
    In particular the Lie affgebras given by, $[a,b]=a$ and $[a,b]=\mathrm{const.}$ are not isomorphic if $A$ has more than one element.
\end{ex}

\begin{prop}\label{prop.action}
Given an affine $\mathbb{K}$-module $A$  and  $\zeta\in \mathbb{K}$, define
\begin{equation}\label{Lie.action}
  [-,-]:A\times A\longrightarrow A, \qquad [a,b] = \zeta \triangleright _a b.  
\end{equation}
Then $[-,-]$ is a Lie bracket on $A$. If $\mathbb{K}$ is a field and $\dim A \geq 1$, then the Lie affgebras corresponding to $\zeta_1$ and $\zeta_2$ are isomorphic if and only if $\zeta_1 = \zeta_2$.
\end{prop}
\begin{proof}
    In view of the condition (a) in Definition~\ref{def.aff.mod} and equation \eqref{heap.middle}, $[-,-]$ is a bi-heap homomorphism. We will use the base change property (c) in Definition~\ref{def.aff.mod} repeatedly to prove that $[-,-]$ is a bi-affine transformation. First,
    $$
    \begin{aligned}
       {}[\alpha\triangleright_ab,c] &= \zeta\triangleright_{\alpha\triangleright_ab}c = \big\langle\zeta\triangleright_a c, \zeta\triangleright_a (\alpha\triangleright_ab), \alpha\triangleright_ab\big\rangle
        =\big\langle\zeta\triangleright_a c, (\zeta\alpha)\triangleright_ab), \alpha\triangleright_ab\big\rangle ,
    \end{aligned}
    $$
    by the associative law (b) in Definition~\ref{def.aff.mod}.
    On the other hand
    $$
    \begin{aligned}
        \alpha\triangleright_{[a,c]}([b,c]) &= \big\langle \alpha\triangleright_a (\zeta\triangleright_bc), \alpha\triangleright_a (\zeta\triangleright_ac), \zeta\triangleright_ac\big\rangle\\
        &= \big\langle (\alpha\zeta)\triangleright_a c, (\alpha\zeta)\triangleright_a b, \alpha\triangleright_ab, (\alpha\zeta)\triangleright_a c, \zeta\triangleright_ac\big\rangle = [\alpha\triangleright_ab,c].
    \end{aligned}
    $$
    Therefore, $[-,c]$ is an affine transformation. The other affine homomorphism property is shown similarly:
    $$
    \begin{aligned}
        {}[a,\alpha\triangleright_bc] &= \zeta\triangleright_a(\alpha\triangleright_bc)\big\rangle=
     \big\langle (\alpha\zeta)\triangleright_a c, (\alpha\zeta)\triangleright_ab , \zeta\triangleright_ab\big\rangle\\
     &= \big\langle\alpha\triangleright_a[a,c], \alpha\triangleright_a[a,b], [a,b]\big\rangle = \alpha\triangleright_{[a,b]}[a,c],
    \end{aligned}
    $$
    as required. Next we need to check the condition \eqref{left.Lie}. Note that, for all $a,b\in A$, $[a,a] = a$ by \eqref{aa}, and hence
    $$
    \big\langle [a,b],[a,a],[b,a]\big\rangle = \big\langle\zeta\triangleright_a b, a, \zeta\triangleright_a a, \zeta\triangleright_a b, b\big\rangle = b = [b,b].
    $$
    Finally, by fixing $o\in A$, and using (c) in Definition~\ref{def.aff.mod} we can express any action $\triangleright_a$ etc.\ in terms of $\triangleright_o$ to obtain
    $$
    [a,[b,c]] = \big\langle\zeta^2 \triangleright_o c, \zeta^2 \triangleright_o b, \zeta \triangleright_o b, \zeta \triangleright_o a,a\big\rangle .
    $$
    The Jacobi identity \eqref{l.Jacobi} can now be verified by varying $a,b,c$ cyclically and using the reshuffling and cancellation rules for the abelian heap operation.

    Now, assume that $\mathbb{K}$ is a field, $\dim A\geq 1$ and consider two Lie brackets on $A$,
    $$
    [a,b]_1 = \zeta_1\triangleright_a b \quad \mbox{and} \quad [a,b]_2 = \zeta_2\triangleright_a b, \qquad \mbox{for all $a,b\in A$}.
    $$
    An affine isomorphism $f:A\to A$ is an  
    isomorphism of Lie affgebras $(A,[-,-]_1) \to (A,[-,-]_2)$ if and only if
    $$
      \zeta_2\triangleright_{f(a)}f(b) = [f(a),f(b)]_2 = f([a,b]_1)= \zeta_1\triangleright_{f(a)}f(b).
    $$
Since $f$ is bijective this is equivalent to the statement that $\zeta_1\triangleright_a b  = \zeta_2\triangleright_a b $, for all $a,b\in A$. In the vector space $A_a$, the above equality implies that
$
\zeta_1b  = \zeta_2 b
$, for all $b$, and since $A_a$ is at least one-dimensional this is equivalent to $\zeta_1=\zeta_2$ as required.
\end{proof}

    The term {\em Lie affgebra} was introduced by Grabowska, Grabowski and Urba\'nski in \cite{GraGra:Lie} to describe a Lie-type structure on an affine space and its relation to Definition~\ref{def.Lie} should be clarified. In \cite{GraGra:Lie} an affine space is understood in the traditional manner as a set $A$ over the vector space $\overset{\rightarrow}{A}$. A Lie bracket on $(A,\overset{\rightarrow}{A})$ is then defined as an anti-symmetric bi-affine map
    $$
    [-,-]_v: A\times A \longrightarrow \overset{\rightarrow}{A},
    $$
    that satisfies the Jacobi identity in $\overset{\rightarrow}{A}$:
    \begin{equation}\label{Jacobi}
        \overset{\rightarrow~~~~~~~~~~}{[a,[b,c]_v]_v} + \overset{\rightarrow~~~~~~~~~~}{[b,[c,a]_v]_v} + \overset{\rightarrow~~~~~~~~~~}{[c,[a,b]_v]_v} = 0,
    \end{equation}
    where the arrow indicates the linearisation of the maps $[a,-]_v$ etc. We will refer to the map $[-,-]_v$ as to the \textbf{vector-valued Lie bracket} on $(A,\overset{\rightarrow}{A})$.

\begin{prop}\label{prop.GGU}
    Let $A$ be an affine space over the field $\mathbb{K}$ of characteristic different from 2. For any $o\in A$, there is a bijective correspondence between idempotent Lie brackets on $A$ and vector-valued Lie brackets on $(A,A_o)$.
\end{prop}
\begin{proof}
    Assume that $[-,-]$ is an idempotent Lie bracket on $A$, that is $[a,a]=a$, for all $a\in A$, and, for all $a,b\in A$, define $[a,b]_v\in A_o$ by
    $$
     [a,b]_v = \langle [a,b],b,o\rangle = [a,b]-b.
    $$
    Then, using the idempotent property of $[-,-]$ and \eqref{anti} we obtain
    $$
    \begin{aligned}
        {}[a,b]_v + [b,a]_v &= [ab]-b+ [b,a] - a = 
        \langle [a,b],[b,b], [b,a]\rangle  - a = [a,a]-a=o,
    \end{aligned}
    $$
    so the bracket is anti-symmetric. The linearisation of $[a,-]_v$ comes out as
    $$
    \begin{aligned}
      \overset{\rightarrow~~~~~~~~~~}{[a,[b,c]_v]_v} &= \overset{\rightarrow~~~~~~~~~~~~~~~}{[a,\langle [b,c],c,o\rangle]_v} = \langle [a,[b,c]]_v, [a,c]_v,o\rangle \\
      &= \langle [a,[b,c]], [b,c],o, [a,c]_v,o\rangle \\
      &= \langle [a,[b,c]], c,o,o,c,[b,c],o, [a,c]_v,o\rangle\\
      &= \langle [a,[b,c]], c,o, [b,c]_v,o,[a,c]_v,o\rangle, 
      \\
      &= \langle [a,[b,c]], [c,c],o, [b,c]_v,o,[a,c]_v,o\rangle\\
      &= [a,[b,c]] - [c,c] -  [b,c]_v - [a,c]_v,
    \end{aligned}
    $$
    where the last expression is written in $A_o$.  The Jacobi identity \eqref{l.Jacobi} written in $A_o$ reads:
    \begin{equation}\label{Jacobi.o}
    [a,[b,c]]- [a,a]+ [b,[c,a]]- [b,b]+ [c,[a,b]] -[c,c] =o,
         \end{equation}
    and its combination with the anti-symmetry of $[-,-]_v$ implies the Jacobi identity \eqref{Jacobi}. 

    In the opposite direction, one sets
    $$
    [a,b] = [a,b]_v +b.
    $$
    This is an idempotent by the antisymmetry of $[-,-]_v$ (and the fact that the characteristic of the field is different from two). Antisymmetry also implies that $[-,-]$ satisfies condition \eqref{anti}. Noting that for an affine map $f:A\to A$ in $(A,A_o)$ and $a\in A_o$ and $b\in A$,
    $$
    f(a+b) = \overset{\to}{f}(a) + f(b),
    $$
    we can compute 
    $$
    \begin{aligned}
        {}[a,[b,c]] &= [a,[b,c]]_v + [b,c] = [a,[b,c]_v + c]_v + [b,c]_v +c \\
        &= \overset{\rightarrow~~~~~~~~~~}{[a,[b,c]_v]_v} +  [a,c]_v + [b,c]_v +[c,c].
    \end{aligned}
    $$
    Now the Jacobi identity \eqref{Jacobi} together with the antisymmetry of $[-,-]_v$ imply the Jacobi identity \eqref{Jacobi.o}.
    \end{proof}

\begin{ex}\label{ex.GGU}
    In view of Proposition~\ref{prop.GGU}, Proposition~\ref{prop.action} gives a family of vector-valued Lie brackets on an affine space $(A,\overset{\rightarrow}{A})$ labelled by scalars $\alpha \in \mathbb{K}$:
    $$
    [a,b]_v = \alpha \overset{\longrightarrow}{ab}, \qquad a,b\in A.
    $$
\end{ex}

Every associative algebra is a Lie algebra with the commutator bracket. An analogous statement is true for associative affgebras.

\begin{prop}\label{prop.comm}
    An associative $\mathbb{K}$-affgebra $A$ is a Lie affgebra with the bracket
    \begin{equation}\label{comm}
           [a,b] = \langle ab,ba,b\rangle,
    \end{equation}
    for all $a,b\in A$. 
\end{prop}
\begin{proof}
   The maps $[a,-]$ and $[-,b]$ induced by the bracket \eqref{comm} are affine homomorphisms by Lemma~\ref{lem.aff.map}.
 
    Checking the anti-symmetry property \eqref{anti} is straightforward. For the Jacobi identity, note that the bracket \eqref{comm} is an idempotent operation and that
    $$
    [a,[b,c]] = \langle abc, acb, ac, bca, cba, ca,bc,cb,c\rangle.
    $$
    Permuting  this cyclically one concludes that the Jacobi identity \eqref{l.Jacobi} holds.
\end{proof}

Again, as is the case for Lie algebras, a commutator Lie bracket needs not be defined on an associative algebra, it is sufficient to consider {\em pre-Lie algebras} introduced in \cite{Vin:the} and \cite{Ger:coh} (see \cite{Man:sho} for a review).

\begin{defn}\label{def.pre}
A {\em left pre-Lie affgebra} is an affine space $A$ together with the bi-affine map $\cdot : A\times A\longrightarrow A$, such that, for all $a,b,c\in A$,
\begin{equation}\label{left.pre}
    (a\cdot b)\cdot c = \big\langle a\cdot (b\cdot c), b\cdot (a\cdot c), (b\cdot a)\cdot c \big\rangle.
\end{equation}

Similarly, a {\em right pre-Lie affgebra} is an affine space with bi-affine binary operation $\cdot$, satisfying the condition
\begin{equation}\label{right.pre}
   a\cdot (b\cdot c)  = \big \langle (a\cdot b)\cdot c, (a\cdot c )\cdot b  , a\cdot (c\cdot b)\big\rangle,
\end{equation}
for all $a,b,c\in A$.
\end{defn}
We note in passing that, when written in terms of addition $a+b = \langle a,o,b\rangle$ the conditions \eqref{left.pre} and \eqref{right.pre} coincide exactly with the corresponding pre-Lie algebras conditions.

\begin{prop}\label{prop.pre}
    Let $(A,\cdot)$ be a right (or left) pre-Lie affgebra. Then $A$ is a Lie affgebra with the bracket
    $$
    [a,b] = \langle a\cdot b,b\cdot a,b\rangle,
    $$
    for all $a,b\in A$.
\end{prop}
\begin{proof}
    The only point worth mentioning here is that, in the case of a right pre-Lie affgebra
    $$
    \begin{aligned}
        {}[a,[b,c]] &= \big\langle a\cdot (b\cdot c), a\cdot (c\cdot b), a\cdot c, (b\cdot c)\cdot a, (c\cdot b)\cdot a, c\cdot a,b\cdot c,c\cdot b,c\big\rangle\\
        &= \big\langle (a\cdot b)\cdot c, (a\cdot c)\cdot b, a\cdot c, (b\cdot c)\cdot a, (c\cdot b)\cdot a, c\cdot a,b\cdot c,c\cdot b,c\big\rangle,
    \end{aligned}
    $$
    by \eqref{right.pre}. Once the brackets are redistributed in a uniform way as above, the cancellation of the cyclic combinations follows by the same arguments as in the associative case. The left pre-Lie affgebra condition \eqref{left.pre} allows one to distribute the brackets in double operation to the form $-\cdot (-\cdot -)$.
\end{proof}

\begin{defn}\label{def.deriv}
Let $A$ be a non-necessarily associative $\mathbb{K}$-affgebra and let $\sigma:A\to A$ be an affine map such that, for all $a,b\in A$,
\begin{equation}\label{sigma}
    \sigma(ab) = \big\langle \sigma (a)b, \sigma(ab), a\sigma (b)\big\rangle.
\end{equation}
A \textbf{derivation along $\sigma$} is an affine homomorphism $X: A\to A$, such that, 
\begin{subequations}\label{deriv}
    \begin{equation}\label{deriv.comm}
    X\sigma = \sigma X,
\end{equation}
\begin{equation}\label{deriv.Leibniz}
X(ab) = \big\langle X(a)b, \sigma(ab), aX(b)\big\rangle, \qquad \mbox{for all $a,b\in A$}.
\end{equation}
\end{subequations}
The set of all derivations along $\sigma$ on $A$ is denoted by $\mathrm{Der}_\sigma(A)$. Note that $\sigma \in \mathrm{Der}_\sigma(A)$.
\end{defn}
\begin{ex}\label{ex.deriv}
   Let $[-,-]$ be the commutator Lie bracket \eqref{comm} on  an associative $\mathbb{K}$-affgebra $A$. Then, for all $a\in A$, the map $X_a:A\longrightarrow A$, $b\longmapsto [a,b]$ is a derivation along the identity on $A$.
\end{ex}
\begin{proof}
    Obviously $X_a$ is an affine transformation. Checking the derivation property (with $\sigma = \mathrm{id}$) is straightforward.
\end{proof}
\begin{ex}\label{ex.abelian}
    Let $A$ be an affine $\mathbb{K}$-module, take $o\in A$ and view it as an associative affgebra with multiplication
    $$
    ab = \langle a,o,b\rangle  = a+b,
    $$
    that is the multiplication of $A$ is addition in the $\mathbb{K}$-module $A_o$ (with the $\mathbb{K}$-action $\alpha a = \alpha\triangleright_oa$). Then $\mathrm{Der}_\id(A) = \mathrm{End}_{\mathbb{K}}(A_o)$.
\end{ex}
\begin{proof}
    An affine endomorphism $X:A\to A$ is a derivation along the identity if and only if
    $$
    X(a+b) = X(ab) = \big\langle X(a), o, b, a, o ,b , a, o, X(b) \big\rangle  = \big\langle X(a), o, X(b) \big\rangle = X(a) +X(b),
    $$
    i.e., $X$ is an additive map. In particular $X(o) =o$, and since $X$ is an affine transformation, for all $\alpha \in \mathbb{K}$ and $a\in A_o$
    $$
    X(\alpha a) = \alpha \triangleright_{X(o)}X(a) = \alpha X(o).
    $$
    Therefore, $X$ is a derivation along the identity if and only if it is an endomorphism of $\mathbb{K}$-modules.
    \end{proof}
\begin{prop}\label{prop.Lie.deriv}
    Let $L$ be a Lie affgebra with an idempotent bracket.
      Then, for all $a\in L$, the map 
    $$
    X_a: L\longrightarrow L, \qquad b\longmapsto [a,b],
    $$
    is a derivation on $L$ along the identity.
\end{prop}
\begin{proof}
    By the definition of a Lie bracket, the map $X_a$ is an affine transformation. Take any $b,c\in L$, and using the Jacobi identity \eqref{l.Jacobi}, the antisymmetry \eqref{anti}, and the idempotent property of the Lie bracket, compute
    $$
    \begin{aligned}
        X_a([b,c]) &= [a,[b,c]] = \Big\langle  [a,a], [b,[c,a]], [b,b], [c,[a,b]], [c,c]\Big\rangle \\
        &= \Big\langle  [a,a], [b,[c,c]], [b,[a,c]], [b,[a,a]], [b,b], [[a,b],[a,b]], [[a,b], c],  [c,c], [c,c]\Big\rangle \\
        &= \Big\langle  a, [b,c], [b,X_a(c)], [b,a], [b,b], [a,b], [X_a(b), c]\Big\rangle\\
        &=\Big\langle  [X_a(b), c], [b,c], [b,X_a(c)], [a,a], a\Big\rangle = \Big\langle  [X_a(b), c], [b,c], [b,X_a(c)]\Big\rangle,
    \end{aligned}
    $$
    as needed.
\end{proof}
\begin{thm}\label{thm.deriv}
Let $A$ be a non-necessarily associative $\mathbb{K}$-affgebra  and let $\sigma:A\to A$ be an affine map. Then
\begin{zlist}
    \item Provided $\sigma$ satisfies \eqref{sigma}, $\mathrm{Der}_\sigma(A)$ is a Lie affgebra with the $\mathbb{K}$-affine structure arising from $\mathrm{Aff}(A)$ and the Lie bracket
\begin{equation}\label{deriv.Lie}
    [X,Y] = \big\langle XY, YX, \sigma \big\rangle,
\end{equation}
    for all $X,Y\in \mathrm{Der}_\sigma(A)$.
    \item If the affine $\mathbb{K}$-submodule
    $$
    \mathrm{Aff}(A)_\sigma := \{X\in \mathrm{Aff}(A)\; |\; \mbox{$X$ satisfies \eqref{deriv.Leibniz}}\} \subseteq \mathrm{Aff}(A),
    $$
    is a Lie affgebra with bracket \eqref{deriv.Lie}, then $\sigma \in \mathrm{Aff}(A)_\sigma$ and $\mathrm{Aff}(A)_\sigma = \mathrm{Der}_\sigma(A)$.
\end{zlist}

\end{thm}
\begin{proof}
    (1) First we need to check whether $\mathrm{Der}_\sigma(A)$ is closed under the $\mathbb{K}$-action \eqref{end.act} (it is clear that $\mathrm{Der}_\sigma(A)$ is a heap with the pointwise operation). In view of the base change property (c) in Definition~\ref{def.aff.mod}, actions with different bases (elements under the symbol $\triangleright$) can be converted to the action with a fixed base. Hence, take $o\in A$, then for all $X,Y\in \mathrm{Der}_\sigma(A)$, $a\in A$ and $\alpha \in \mathbb{K}$,
    $$
    (\alpha\triangleright_X Y)(a) =  \big\langle \alpha \triangleright_o Y(a),  \alpha \triangleright_o X(a), X(a)\big\rangle .
    $$
    To simplify notation further we will denote $\triangleright_o$ by $\cdot$. We compute, for all $a,b\in A$,
    $$
    \begin{aligned}
       (\alpha\triangleright_X Y)(ab) &=\big\langle \alpha \cdot Y(ab),  \alpha \cdot X(ab), X(ab)\big\rangle \\
       &= \big\langle \alpha \cdot (Y(a)b),  \alpha \cdot \sigma(ab), \alpha \cdot (aY(b)),\\
       & ~~~~~\alpha \cdot (X(a)b),  \alpha \cdot \sigma(ab), \alpha \cdot (aX(b)), X(a)b, \sigma(ab), a X(b) \big\rangle\\
       &= \big\langle (\alpha \cdot Y(a))b, (\alpha \cdot X(a))b , X(a)b, \sigma(ab), a\alpha \cdot Y(b),
         a\alpha \cdot X(b), a X(b) \big\rangle\\
         &= \Big\langle \big\langle\alpha \cdot Y(a), \alpha \cdot X(a) , X(a)\big\rangle b, \sigma(ab), a\big\langle\alpha \cdot Y(b),
         \alpha \cdot X(b), X(b) \big\rangle\Big\rangle\\
         &= \big\langle(\alpha\triangleright_X Y)(a)b, \sigma(ab), a(\alpha\triangleright_X Y)(b)\big\rangle,
    \end{aligned}
    $$
    which proves that $\alpha\triangleright_X Y$ is a derivation along $\sigma$ as required.

    Next we need to check that $[X,Y]$ is a derivation along $\sigma$. Since all the maps involved in the definition of 
    $[X,Y]$ are  homomorphisms of affine modules,  $[X,Y]$ is an affine module homomorphism as well by Lemma~\ref{lem.aff.map}. 
      The derivation property is checked as follows. For all $a,b\in A$,
    $$
    \begin{aligned}
        {}[X,Y]&(ab) = \big\langle XY(ab), YX(ab), \sigma(ab)\big\rangle\\
        &= \big\langle X(aY(b)), X\sigma(ab), X(Y(a)b), Y(aX(b)), Y\sigma (ab), Y(X(a)b), \sigma (ab)\big\rangle\\
        &= \big\langle X(a)Y(b),\sigma(aY(b)), aXY(b), X\sigma(ab), XY(a)b, \sigma (Y(a)b), Y(a)X(b),\\
        &
        Y(a)X(b),  \sigma(aX(b)), aYX(b), Y\sigma (ab), YX(a)b, \sigma(X(a) b), X(a)Y(b), \sigma (ab)\big\rangle\\
        &=\big\langle  \sigma (ab),\sigma(aY(b)), aXY(b), X\sigma(ab), XY(a)b, \sigma (Y(a)b), \\
        &
          \sigma(aX(b)), aYX(b), Y\sigma (ab), YX(a)b, \sigma(X(a) b)\big\rangle\\
        &=\big\langle  \sigma (ab),\sigma Y(ab), a[X,Y](b), X\sigma(ab), [X,Y](a)b, \sigma^2(ab) \\
        &
          \sigma X(ab), a\sigma(b), Y\sigma (ab), \sigma(a)b, \sigma^2(ab)\big\rangle
          \\
        &=\big\langle  a[X,Y](b), a\sigma(b), \sigma (ab),  \sigma(a)b, [X,Y](a)b  \big\rangle \\
        &= \big\langle  a[X,Y](b),  \sigma (ab),   [X,Y](a)b  \big\rangle .
    \end{aligned}
    $$
    Here we used the derivation property in obtaining the second, third, fourth and the fifth, equalities, equation \eqref{deriv.comm} to obtain the sixth one and \eqref{sigma} to derive the last equality. We also used freely the cancellation and rearranging rules stemming from the definition of an abelian heap. This proves that $[X,Y]$ is a derivation along $\sigma$.

    The functions $[-,Y]$ and $[X,-]$ are  affine maps by Lemma~\ref{lem.aff.map}.
    It remains to check that $[-,-]$ is a Lie bracket. First, for all $X,Y\in \mathrm{Der}_\sigma(A)$, note that $[X,X] = \sigma =[Y,Y]$, and hence 
    $$
    \big\langle[X,Y],[X,X],[Y,X]\big\rangle = \big\langle XY,YX,\sigma,\sigma,YX,XY\big\rangle = \sigma = [Y,Y].
    $$
    Finally, taking in addition $Z\in \mathrm{Der}_\sigma(A)$,
    $$
    \begin{aligned}
        \big\langle[X,[Y,Z]],&[X,X],[Y,[Z,X]], [Y,Y], [Z,[X,Y]]\big\rangle \\
        &= \big\langle X[Y,Z],[Y,Z]X,\sigma, \sigma, Y[Z,X],[Z,X]Y,\sigma, \sigma, Z[X,Y],[X,Y]Z,\sigma\big\rangle
        \\
        &= \big\langle XYZ, XZY, X\sigma,\sigma X, ZYX, YZX, YZX, YXZ,Y\sigma, \sigma Y,
        \\
        &~~~~XZY, ZXY, ZXY, ZYX,Z \sigma, \sigma Z , YXZ, XYZ,\sigma\big\rangle = \sigma = [Z,Z],
    \end{aligned}
    $$
    where the condition \eqref{deriv.comm} was used. This completes the proof that the set of all derivations on $A$ along $\sigma$ is a Lie affgebra.

    (2) Since $[X,X]=\sigma$ and $\mathrm{Aff}(A)_\sigma$ is closed under the Lie bracket, $\sigma \in \mathrm{Der}_\sigma(A)$. Furthermore, exploring the Jacobi identity for $X, Y=Z \in \mathrm{Aff}(A)_\sigma$ we find,
    $$
    \langle \sigma, YX^2,\sigma X, X\sigma, YX^2 \rangle = \sigma,
    $$
    which implies that $X\sigma = \sigma X$ and hence every element of $\mathrm{Aff}(A)_\sigma$ is a derivation along $\sigma$.
\end{proof}

The following theorem reveals that just as any associative truss can be retracted to an associative ring \cite[Theorem~4.3]{AndBrzRyb:ext}, any Lie affgebra can be retracted to a Lie algebra.

\begin{thm}\label{thm.Lie.alg}
Let $(A,[-,-])$ be a Lie $\mathbb{K}$-affgebra. Then, for all $o\in A$ the $\mathbb{K}$-module $A_o$ together with the bracket
\begin{equation}\label{Lie.bracket}
    [a,b]_{o}= \langle [a,b],[a,o],[o,o],[o,b],o \rangle = [a,b]-[a,o]+[o,o]-[o,b],
\end{equation}
is a Lie $\mathbb{K}$-algebra. Furthermore, any two Lie $\mathbb{K}$-algebras $(A_o,[-,-]_o)$ and $(A_u,[-,-]_u)$ are mutually isomorphic.
\end{thm}
\begin{proof} 
The bracket $[-,-]_o$ is a bi-linearisarion of the bi-affine bracket $[-,-]$ by the repetitive use of the linearisation formula \eqref{linearisation} (step-by-step for each argument), and hence is a bilinear operation on $A_o$.

Written in terms of addition and subtraction in $A_o$ the antisymmetry property \eqref{anti} and Jacobi identity \eqref{l.Jacobi} of the Lie bracket  $[-,-]$, for all  $a,b,c \in A_o$ come out as
$$[a,b]-[a,a]+[b,a]=[b,b],\quad 
[a,[b,c]]-[a,a]+[b,[c,a]]-[b,b]+[c,[a,b]]=[c,c].
$$
The first equality immediately implies that $[a,a]_o=o$ and 
$[a,b]_o = - [b,a]_o.$ 

The second equality is used to check the Jacobi identity for $[-,-]_o$,
\begin{align*}
[a,[b,c]_o ]_o + \mathrm{cycl.}&= [a,[b,c]-[b,o]+[o,o]-[o,c]]_o + \mathrm{cycl.}, \\
&= [a,[b,c]]_o - [a,[b,o]]_o +[a,[o,o]]_o - [a,[o,c]]_o + \mathrm{cycl.}, \\
&=[a,[b,c]] - [o,[b,c]] - [a,[b,o]] + [o,[b,o]] \\
&+ [a,[o,o]] - [o,[o,o]] - [a,[o,c]] + [o,[o,c]] + \mathrm{cycl.} = o.
\end{align*}

Finally, using the translation isomorphism $\tau^u_o$ from Remark~\ref{rem.trans}, we obtain the following equality of terms (binary operations in $A_u$),
    \begin{align*}
        [\tau^u_o(a), \tau^u_o (b) ]_u &= [a-o,b-o]_u, \\
&= [a,b]_u- [a,o]_u-[o,b]_u +[o,o]_u, \\
        &= \langle [a,b],[a,o],[o,o],[o,b],u \rangle = \tau^u_o ([a,b]_o).
    \end{align*}
This proves the isomorphism of  Lie algebras.
\end{proof}

\begin{ex}\label{ex.comm.o}
Let $A$ be an associative $\mathbb{K}$-affgebra with the commutator bracket \eqref{comm} in Proposition~\ref{prop.comm}. Take any $o\in A$. By \cite[Theorem~4.3]{AndBrzRyb:ext} $A_o$ is an associative algebra with the multiplication
$$
a\bullet b = ab - ao +o^2 -ob.
$$
One easily checks that 
$$
[a,b]_o = a\bullet b - b\bullet a,
$$
 the usual commutator Lie bracket on an associative algebra.   
\end{ex}
\begin{ex}\label{ex.action.o}
Let $A$ be a Lie affgebra with the bracket $[a,b]= \zeta\triangleright_ab$, $\zeta\in \mathbb{K}$ (see Proposition~\ref{prop.action}). Take any $o\in A$. Then, irrespective of the choice of $\zeta$,
$$
[a,b]_o = o,
$$
the trivial Lie bracket on $A_o$. As shown in Proposition~\ref{prop.action} in case $\mathbb{K}$ a field the Lie brackets on $A$ depend heavily on the choice of $\zeta$. Thus non-isomorphic Lie affgebras can reduce to isomorphic (or even identical) Lie algebras.

In a similar way all Lie affgebra structures $[a,b]=\sigma(a)$ discussed in Example~\ref{ex.trivial} collapse to the zero Lie bracket on $A_o$.  
\end{ex}

\section{Nijenhuis operators on Lie affgebras}\label{sec.Nij}
We now look to how Nijenhuis operators behave on Lie affgebras, with adapted definitions from \cite{KosSchMag:Poi}. In particular, we show that Nijenhuis operators give rise to a family of compatible Lie brackets on an affine space.
\begin{defn}\label{def.nijen.op}
Let $(A,[-,-])$ be a Lie affgebra. An affine homomorphism $N: A \rightarrow A$ is called a \textbf{Nijenhuis operator} if, 
for all $a,b \in A$,
\begin{equation}\label{Nij.cond}
[N(a),N(b)] = N \bigl( \langle [N(a),b] , N([a,b]) , [a,N(b)]  \rangle \bigr).
\end{equation}
The binary operation $[-,-]_N$ on $A$, given by, 
\begin{equation}\label{Nij.bra}
 [a,b]_N = \langle [N(a),b] , N([a,b]) , [a,N(b)]  \rangle,   
\end{equation}
is called the \textbf{Nijenhuis bracket}.
\end{defn}

Note that, since $N$ is an affine homomorphism, the Nijenhuis bracket $[-,-]_N$ is affine on each argument by Lemma~\ref{lem.aff.map} and hence a bi-affine map.
\begin{rem}\label{rem.Nij}
To make the following calculations more legible, from this point onward we will be using the operator notation convention $N(a) = Na $. We will also use the additive notation for ternary heap operation \eqref{+-}. In this notation the Nijenhuis bracket \eqref{Nij.bra}  and the Nijenhuis operator condition \eqref{Nij.cond} come out as
$$
[a,b]_N =  [Na,b] - N[a,b]+ [a,Nb], \qquad [Na,Nb] = N[a,b]_N.
$$
\end{rem}

\begin{thm}\label{nij.is.lie.bra}
A Nijenhuis bracket on a Lie affgebra  $(A,[-,-])$ is a Lie bracket. 
\begin{proof}
First, we see if the Nijenhuis bracket satisfies the antisymmetry property. We start with
\begin{align*}
\langle [b,b]_N, [b,a]_N,[a,a]_N\rangle &= [Nb,b]- N[b,b] +[b,Nb] \\
&- [Nb,a] + N[b,a] - [b,Na] \\
&+ [Na,a] - N[a,a] + [a,Na],
\end{align*}
and then use the antisymmetry property \eqref{anti} of $[-,-]$,
\begin{align*}
\langle [b,b]_N, &[b,a]_N,[a,a]_N\rangle = [Nb,b]- N[b,b] +[b,Nb] 
- [a,a] + [a,Nb] - [Nb,Nb]+N[a,a] \\
&-N[a,b]+N[b,b] 
- [Na,Na] + [Na,b] - [b,b] 
+ [Na,a] - N[a,a] + [a,Na].
\end{align*}
Next we perform cancellations and apply again the antisymmetry of $[-,-]$,
$$ \langle [b,b]_N, [b,a]_N,[a,a]_N\rangle = [Na,b]- N[a,b] +[a,Nb] = [a,b]_N, $$
thus the antisymmetry property holds for $[-,-]_N$. 

It remains to  check that the Nijenhuis bracket satisfies the Jacobi identity. To this end we start with developing the left-hand side of the Jacobi identity 
$$ 
\begin{aligned} 
\mathrm{LHS} &= \langle [a,[b,c]_N]_N,  [a,a]_N, [b,[c,a]_N]_N, [b,b]_N, [c,[a,b]_N]_N \rangle  \\
&= [Na,[Nb,c]] - [Na,N[b,c]] 
+ [Na,[b,Nc]] -N[a,[Nb,c]] + N[a,N[b,c]] \\
&- N[a,[b,Nc]] + [a,[Nb,Nc]]
-[a,a] + N[a,a] - [Na,Na] 
+[Nb,[Nc,a]]\\
&- [Nb,N[c,a]] + [Nb,[c,Na]] 
-N[b,[Nc,a]] + N[b,N[c,a]] 
- N[b,[c,Na]]\\
&+ [b,[Nc,Na]]-[b,b]  + N[b,b]
- [Nb,Nb] 
+[Nb,[Nc,a]] - [Nb,N[c,a]]\\
&+ [Nb,[c,Na]] -N[b,[Nc,a]]
+ N[b,N[c,a]] 
- N[b,[c,Na]] + [b,[Nc,Na]],
\end{aligned}
$$
and then apply the Jacobi identity for $[-,-]$ several times to obtain
\begin{align*}
\mathrm{LHS}=&[c,c]+[Na,Na]+[Nc,Nc]+[Nb,Nb]+[Nc,Nc]  
+ N[Na,[b,c]] - N[Na,Na]\\
&- N[b,b] - N[c,c] 
+ N[Nb,[c,a]] 
-N[a,a] - N[Nb,Nb] - N[c,c]\\
&+ N[Nc,[a,b]] 
- N[a,a] - N[b,b] - N[Nc,Nc] 
-[Na,N[b,c]]+N[a,N[b,c]]\\
&+N[a,a] 
-[Nb,N[c,a]]+N[b,N[c,a]]+N[b,b] 
-[Nc,N[a,b]]+N[c,N[a,b]].
\end{align*}
In the next step we  perform cancellations and substitutions of the form
$$N^2 [x,y]= N[Nx,y] - [Nx,Ny]+N[x,Ny]. $$
This results in
\begin{align*}
\mathrm{LHS}&= [c,c]+[Na,Na]+[Nc,Nc]+[Nb,Nb]+[Nc,Nc] 
-N[Na,Na]\\
&-N[a,a]-N[Nb,Nb]-N[b,b]-N[Nc,Nc]-N[c,c]-N[c,c] \\
&+N^2 [a,[b,c]] + N^2[b,[c,a]] + N^2 [c,[a,b]].
\end{align*}
The application of  the Jacobi identity for $[-,-]$ yields,
\begin{align*}
\mathrm{LHS}&= [c,c]+[Na,Na]+[Nc,Nc]+[Nb,Nb]+[Nc,Nc] 
-N[Na,Na]-N[a,a]\\
&-N[Nb,Nb]-N[b,b]-N[Nc,Nc]-N[c,c]-N[c,c] \\
&+N^2[a,a]+N^2[b,b]+N^2[c,c].
\end{align*}
It remains to use the definition of the  Nijenhuis bracket \eqref{Nij.bra}, antisymmetry and the Nijenhuis operator \eqref{Nij.cond} condition,  to coclude that 
$$
\mathrm{LHS}= [c,c]-N[c,c]+[Nc,Nc] =[Nc,c] - N[c,c]+ [c,Nc]  = [c,c]_N. 
$$
i.e.
$$ \langle [a,[b,c]_N]_N, [a,a]_N, [b,[c,a]_N]_N, [b,b]_N, [c,[a,b]_N]_N \rangle = [c,c]_N,
$$
as required.
\end{proof}
\end{thm}

\begin{ex}\label{ex.Nij.com}
   Let $A$ be an associative $\mathbb{K}$-affgebra $A$, and  $N:A \to A$ an affine map such that, for all $a,b\in A$,
   $$
   N(a)N(b) = N\langle N(a)b, N(ab), aN(b)\rangle;
   $$
   see \cite{BrzPap:aff}. A straightforward computation
   \begin{align*}
        N \langle [N(a),b],N[a,b],[a,N(b)] \rangle &= N(N(a)b-bN(a)+b-N(ab)+N(ba)-N(b) \\
        &+aN(b)-N(b)a+N(b)), \\ 
        &=N(a)N(b)-N(b)N(a)+N(b)=[N(a),N(b)], 
    \end{align*}
    affirms that 
  $N$ is a Nijenhuis operator on the Lie affgebra $A$ with the bracket given in Proposition~\ref{prop.comm}.
\end{ex}

\begin{ex}\label{ex.Nij.act}
    Let $A$ be the $\mathbb{K}$-affine module with the Lie bracket $[a,b] = \zeta\triangleright_ab$ as in Proposition~\ref{prop.action}. For any affine endomorphism $N:A\longrightarrow A$, using the homomorphism property of $N$ and the base change propoerty (c) in Definition~\ref{def.aff.mod} we find
    $$
    \begin{aligned}
        {}[a,b]_N &= \langle \zeta\triangleright_{Na} b, N(\zeta\triangleright_ab), \zeta\triangleright_aNb\rangle\\
        &= \langle \zeta\triangleright_{a} b, \zeta\triangleright_{a} Na, Na, \zeta\triangleright_{Na} Nb,  \zeta\triangleright_aNb\rangle\\
        &= \langle \zeta\triangleright_{a} b, \zeta\triangleright_{a} Na, Na, Na, \zeta\triangleright_{a} Na,  \zeta\triangleright_aNb,  \zeta\triangleright_aNb\rangle\\
        &= \zeta\triangleright_{a} b = [a,b].
    \end{aligned}
    $$
    Again, by using the homomorphism property of $N$ we immediately find that 
    $$
    [Na,Nb]=N[a,b] = N[a,b]_N.
    $$ Therefore, any affine endomorphism $N$ of $A$ is a Nijenhuis operator for $(A,[-,-])$ and the Nijenhuis Lie bracket coincides with the original Lie bracket on $A$.
\end{ex}

\begin{thm}\label{thm.Nij.iter}
 Let $N$ be a Nijenhuis operator on a Lie $\mathbb{K}$-affgebra $(A,[-,-])$. For any $k\in \mathbb{N}$ we denote by $N^k$ the $k$-th composition power of $N$, with the convention that $N^0=\id$. For all $k,l\in \mathbb{N}$:
 \begin{zlist}
         \item $N^k$ is a Nijenhuis operator on $(A,[-,-])$.
         \item
                  $N^l$ is a Nijenhuis operator on $(A,[-,-]_{N^k})$.
                  \item For all $\alpha \in  \mathbb{K}$, the map
         $$
         N_\alpha^{k,l}:=\alpha\triangleright_{N^k}N^l:  A\longrightarrow A, \qquad a \longmapsto \alpha\triangleright_{N^ka}{N^la},
         $$
     is a Nijenhuis operator on $(A, [-,-])$.
 \end{zlist}
\begin{proof} Using similar methodology to that of Theorem 4.9 in \cite{BrzPap:aff} we first show that, for all $a,b \in A, k,l \in \mathbb{N}$, 
\begin{align}
[N^k a, Nb]&= \langle N[N^ka,b],N^{k+1}[a,b],N^k[a,Nb] \rangle, \label{induct:4.1} \\
[Na, N^k b] &= \langle N[a,N^k b],N^{k+1}[a,b],N^k[Na,b] \rangle, \label{induct:4.2} \\
N^{k+1}[a,b] &= \langle N[N^k a,b],[N^{k}a,Nb],N^k[a,Nb] \rangle, \label{induct:4.3} \\
&= \langle N^k [Na,b],[Na,N^k b],N[a,N^k b] \rangle. \label{induct:4.4}
\end{align}
We will show how (\ref{induct:4.1}) can be proven by induction and only note that (\ref{induct:4.2}) can be proven symmetrically, and that proofs of (\ref{induct:4.3}) and  (\ref{induct:4.4}) follow trivially.
Starting with (\ref{induct:4.1}), for $k=0$ we have $N^0 = \id$, and for $k=1$ we have the Nijenhuis operator. If we assume that equation (\ref{induct:4.1}) holds for $k$, then the Nijenhuis condition yields,
\begin{align}
    [N^{k+1}a,Nb] &= [N(N^k a), N b] = N \langle [N^{k+1} a, b], N[N^k a,b],[N^k a, Nb] \rangle,  \\
    &= N \langle [N^{k+1} a, b], N[N^k a,b], N[N^ka,b],N^{k+1}[a,b],N^k[a,Nb]\rangle, \\
    &= \langle N[N^{k+1}a,b],N^{k+2}[a,b],N^{k+1}[a,Nb] \rangle. \label{equat:k+1}
\end{align}
Thus, (\ref{induct:4.1}) holds for all $k \in \mathbb{N}$, by induction. Then using the above, we compute,
\begin{align*}
    [Na,Nb]_{N^k} = \langle &[N^{k+1}a,Nb],N^k [Na,Nb],[Na,N^{k+1}b] \rangle, \\
    = \langle &N[N^{k+1}a,b],N^{k+2}[a,b],N^{k+1}[a,Nb], \\
    &N^{k+1}[Na,b] , N^{k+2} [a,b] , N^{k+1}[a,Nb],\\
    &N^{k+1}[Na,b],N^{k+2}[a,b],N[a,N^{k+1}b] \rangle, \\
    = \langle &N[N^{k+1}a,b], N^{k+2}[a,b],N[a,N^{k+1}b] \rangle = N[a,b]_{N^{k+1}}.
\end{align*}
Now we will show that,
\begin{equation}
    N^l [a,b]_{N^{k+l}} = [N^l a, N^l b]_{N^{k}}. \label{equal:4.8}
\end{equation}
This is proven through a straightforward induction, and we need only demonstrate the case where $l=1$. Using (\ref{induct:4.3}) and (\ref{induct:4.4}), we find,
\begin{align*}
    \langle [Na,b]_{N^k}, & N[a,b]_{N^k}, [a,Nb]_{N^k} \rangle = \langle [N^{k+1} a,b], N^k [Na,b], [Na,N^k b], \\
    &N[N^k a, b], N^{k+1} [a,b], N[a,N^k b], 
    [N^k a,Nb], N^k [a,Nb], [a,N^{k+1} b] \rangle, \\
    = & \langle [N^{k+1} a , b], N^{k+1} [a,b], N^{k+1} [a,b],N^{k+1} [a,b], [a,N^{k+1} b] \rangle\\
    = & \langle [N^{k+1} a , b], N^{k+1} [a,b], [a,N^{k+1} b] \rangle =[a,b]_{N^{k+1}}.
\end{align*}
Lastly, we must demonstrate,
\begin{equation}
    [N^{k,l}_\alpha a, N^{k,l}_\alpha b] = N^{k,l}_\alpha \langle [N^{k,l}_\alpha a,b],N^{k,l}_\alpha [a,b],[a, N^{k,l}_\alpha b]\rangle. \label{equal:4.9}
\end{equation}
To make the following of the argument easier, we use the additive notation explained in Remark~\ref{rem.Nij}, write the affine action relative to the element determining the binary operations as in Remark~\ref{rem.act.+-},   and we compute,
\begin{align*}
[N^{k,l}_\alpha a, N^{k,l}_\alpha b] &= \alpha^2[ N^l a,N^l b] - \alpha^2 [ N^l a, N^k b] + \alpha [N^l a , N^k b] \\
&-\alpha^2 [ N^k a, N^l b] + \alpha^2 [ N^k a, N^k b] - \alpha[ N^k a , N^k b] \\
&+\alpha[N^k a, N^l b] - \alpha[N^k a, N^k b] + [N^k a , N^k b].
\end{align*}
Now compute, applying the Nijenhuis condition, 
\begin{align*}
    N^{k,l}_\alpha & \langle [N^{k,l}_\alpha a,b],N^{k,l}_\alpha [a,b],[a, N^{k,l}_\alpha b]\rangle = \alpha^2 N^l [ N^l a,b] - \alpha^2 N^l [N^k a,b] + \alpha N^l [N^k a,b] \\
    &-\alpha^2 N^k [ N^l a,b] + \alpha^2 N^k [N^k a,b] - \alpha N^k [N^k a,b] 
    +\alpha N^k [ N^l a,b] - \alpha N^k [N^k a,b]  \\
    &+ N^k [N^k a,b]- \alpha^2 N^l N^l [a,b] + \alpha^2 N^l N^k [a,b] - \alpha N^l N^k [a,b] + \alpha^2 N^k N^l [a,b]  \\
    & - \alpha^2 N^k N^k[a,b] + \alpha N^k N^k [a,b] 
    -\alpha N^k N^l [a,b] + \alpha N^k N^k [a,b] - N^k N^k [a,b] \\
    &+ \alpha^2 N^l [a, N^l b]- \alpha^2 N^l [a, N^k b]+ \alpha N^l [a , N^k b] - \alpha^2 N^k [a, N^l b]+ \alpha^2 N^k [a, N^k b]\\
    &- \alpha N^k [a , N^k b] 
    + \alpha N^k [a, N^l b]- \alpha N^k [a, N^k b]+ N^k [a , N^k b], \\ 
    &=- \alpha^2 N^l [N^k a,b] + \alpha^2 N^l N^k [a,b] - \alpha^2 N^l [a, N^k b] + \alpha N^l [N^k a,b] - \alpha N^l N^k [a,b] \\
    & + \alpha N^l [a , N^k b]
    -\alpha^2 N^k [ N^l a,b]+ \alpha^2 N^k N^l [a,b]- \alpha^2 N^k [a, N^l b] +\alpha N^k [ N^l a,b]  \\
    & -\alpha N^k N^l [a,b] + \alpha N^k [a, N^l b]
     +\alpha^2 [N^l a , N^l b] + \alpha^2 [N^k a,N^k b]+ [N^k a, N^k b] \\
    &- 2 \alpha [N^k a, N^k b].
\end{align*}
To see how the terms of our two equations compare, we require some further calculation, for the case where $ k \leq l$ (or similarly where $k \geq l$ ) we use equation (\ref{equal:4.8}), the definition of the Nijenhuis bracket and part (1) of this theorem to see, 
\begin{align*}
    N^k [a,b]_{N^l} &= [N^k a, N^k b]_{N^{l-k}} 
    = [N^k a, N^l b] - N^{l-k}[N^k a, N^k b] + [N^l a, N^k b] \\
    &= [N^k a, N^l b] - N^l [N^k a,b] + N^{l+k} [a,b] - N^l [a,N^k b] + [N^l a, N^k b] \\
    &= [N^k a , N^l b] - N^l [a,b]_k + [N^l a, N^k b],
\end{align*}
thus, we can use the equality, 
\begin{align*}
[N^k a, N^l b]+[N^l a,N^k b] &= N^k ( [N^l a,b] - N^l[a,b]+ [a,N^l b] ) \\
&+ N^l ( [N^k a,b]-N^k[a,b]+[a,N^k b] ),
\end{align*}
to see that we have satisfied the equality (\ref{equal:4.9}).
\end{proof}
\end{thm}

\begin{cor}\label{cor.comp}
 Let $N$ be a Nijenhuis operator on a  Lie $\mathbb{K}$-affgebra $(A,[-,-])$. Then, for all $k,l\in \mathbb{N}$ the  Lie brackets   $[-,-]_{N^k}$ and $[-,-]_{N^l}$ are compatible in the sense that for all $\alpha \in \mathbb{K}$, the map
 \begin{equation}\label{bra.com}
    A\times A\to A, \qquad (a,b)\mapsto \alpha\triangleright_{[a,b]_{N^k}}[a,b]_{N^l},
 \end{equation}
 is a Lie bracket on $A$.
\end{cor}
\begin{proof}
    One easily checks that the operation \eqref{bra.com} is the Nijenhuis bracket corresponding to $N_\alpha^{k,l}$ in Theorem~\eqref{thm.Nij.iter}, and hence it is a Lie bracket by the statement (3) of Theorem~\eqref{thm.Nij.iter} and by Theorem~\ref{nij.is.lie.bra}.
\end{proof}

\begin{ex}\label{ex.comp}
Let $A$ be an affine space and $\sigma \in \mathrm{Aff}(A)$. Take the Lie bracket $[a,b]=\sigma(a)$ of Example~\ref{ex.trivial}. An affine map $N:A\to A$ is a Nijenhuis operator for this bracket provided $N\circ\sigma = \sigma\circ N$. The corresponding Nijenhuis bracket is $[a,b]_N=\sigma(Na)$. By Corollary~\ref{cor.comp}, we have a family of pairwise compatible Lie structures that inculdes the original bracket, $[a,b]_{N^k}=\sigma(N^ka)$. Of course this does not an exhaust the list of compatible structures. Let us take any $\xi \in \mathrm{Aff}(A)$ and consider $[a,b]_\xi = \xi(a)$. Then, for all $a,b\in A$ and $\alpha\in \mathbb{K}$,
$$
\alpha\triangleright_{[a,b]}[a,b]_\xi = \alpha\triangleright_{\sigma(a)}\xi(a) = (\alpha\triangleright_{\sigma}\xi)(a),
$$
which is the Lie bracket of the type Example~\ref{ex.trivial} with $\sigma$ replaced by the affine homomorphism $\alpha\triangleright_{\sigma}\xi$.
\end{ex}

Finally, we prove that Nijenhuis operators on Lie affgebras retract Nijenhuis operators on Lie algebras, that to the linearisation of  a Nijenhuis operator on a Lie affgebra is a Nijenhuis operator on the corresponding Lie algebra.
\begin{thm}\label{thm.nij.final}
Let $(A,[-,-])$ be a Lie affgebra and $N$ a Nijenhuis operator. If we have an element $o \in A$ which acts as the neutral element of addition, then we have the Lie algebra $(A_o, [-,-]_o)$ as seen in Theorem~\ref{thm.Lie.alg}. Let $N_o : A_o \rightarrow A_o $ be defined by, $$ N_o a = \langle Na, No, o \rangle  = Na - No,$$
then $N_o$ is a Nijenhuis operator on the Lie algebra $(A_o, [-,-]_o)$.
\begin{proof}
The function $N_o$ is a linearisation of the affine map $N$ (see \eqref{linearisation}) and hence a linear operator on the vector space $A_o$. 
We must demonstrate that the following equality holds,
\begin{equation}
[N_o a, N_o b]_o = N_o \left([N_o a,b]_o - N_o [a,b]_o + [a,N_o b]_o \right). \label{nij.con.4.7.}
\end{equation}
Compute the left hand side,
    \begin{align*}
    [N_o a, N_o b]_o &= [N_o a , N_o b]-[N_o a, o ]+[o,o]-[o, N_o b] \\ 
    &=[Na,Nb]-[Na,No]+[Na,o]-[No,Nb]+[No,No]-[No,o]\\
    &+[o,Nb]-[o,No]+[o,o] - [Na,o] + [No,o] - [o,o] + [o,o] \\
    &-[o,Nb]+[o,No]-[o,o] \\
    &= [Na,Nb]-[Na,No] +[No,No] -[No,Nb]. 
\end{align*}
Then compute the right hand side,
$$
\begin{aligned}
N_o &\left( [N_o a,b]_o - N_o [a,b]_o + [a,N_o b]_o \right)\\
&= N_o ( [Na,b]_o -[No,b]_o+[o,b]_o -N[a,b]_o + No +[a,Nb]_o - [a,No]_o + [a,o]_o )\\
    &= N ( [Na,b] - [Na,o] + [o,o] - [o,b] - [No,b] + [No,o] - [o,o] + [o,b] \\
    &+[o,b] - [o,o] + [o,o] - [o,b] -N[a,b] + N[a,o] - N[o,o] + N[o,b] \\
    &+[a,Nb]-[a,o]+[o,o]-[o,Nb]-[a,No]+[a,o]-[o,o]+[o,No] \\
    &+[a,o]-[a,o]+[o,o]-[o,o]) -No\\ 
    &= N([Na,b] -N[a,b] +N[a,Nb]) - N([Na,o] -N[a,o] +N[a,No])\\
    &+ N([No,o] -N[o,o] +N[o,No]) - N([No,b] -N[o,b] +N[o,Nb])
      \end{aligned}
    $$
We may then use the Nijenhuis condition when comparing terms to see that both sides of equation (\ref{nij.con.4.7.}) are equal.
\end{proof}
\end{thm}

\section*{Acknowledgements} 

The research of Tomasz Brzezi\'nski is partially supported by the National Science Centre, Poland, grant no. 2019/35/B/ST1/01115.

\end{document}